\documentclass[leqno]{amsart}
\usepackage{times}
\usepackage{amsfonts,amssymb,amsmath,amsgen,amsthm}
\usepackage{hyperref}
\usepackage{color}
\newcommand{\msc}[2][2000]{%
  \let\@oldtitle\@title%
  \gdef\@title{\@oldtitle\footnotetext{#1 \emph{Mathematics subject
        classification.} #2}}%
}

\theoremstyle{plain}
\newtheorem{theorem}{Theorem}[section]

\newtheorem{lemma}[theorem]{Lemma}
\newtheorem{corollary}[theorem]{Corollary}
\newtheorem{proposition}[theorem]{Proposition}
\newtheorem{hyp}[theorem]{Assumption}
\theoremstyle{remark}
\newtheorem{remark}[theorem]{Remark}

\def\C{{\mathbb C}}
\def\R{{\mathbb R}}
\def\T{{\mathbb T}}

\def\({\left(}
\def\){\right)}
\def\<{\left\langle}
\def\>{\right\rangle}
\def\le{\leqslant}
\def\ge{\geqslant}

\def\Tend#1#2{\mathop{\longrightarrow}\limits_{#1\rightarrow#2}}

\def\d{{\partial}}
\def\eps{\varepsilon}

\def\si{{\sigma}}

\DeclareMathOperator{\RE}{Re}
\DeclareMathOperator{\IM}{Im}
\DeclareMathOperator{\diver}{div}

\numberwithin{equation}{section}

\begin{document}

\title[Scattering for NLS with 1D confinement]{Scattering for the
  nonlinear Schr\"odinger equation with a general 
  one-dimensional confinement}

\author[R. Carles]{R\'emi Carles}
\address{CNRS \& Univ. Montpellier\\Math\'ematiques
\\CC~051\\34095 Montpellier\\ France}
\email{Remi.Carles@math.cnrs.fr}

\author[C. Gallo]{Cl\'ement Gallo}
\address{CNRS \& Univ. Montpellier\\Math\'ematiques
\\CC~051\\34095 Montpellier\\ France}
\email{Clement.Gallo@univ-montp2.fr}

\begin{abstract}
We consider the defocusing nonlinear Schr\"odinger equation in several
space dimensions, in the presence of an external potential depending
on only one space variable. This potential is bounded from below, and
may  grow arbitrarily fast at infinity. We prove existence and
uniqueness in the associated Cauchy problem, in a suitable functional
framework, as well as the existence of wave operators when the power
of the nonlinearity is sufficiently large. Asymptotic completeness
then stems from at least two approaches, which are briefly recalled. 
\end{abstract}
\thanks{This project was supported by the French ANR projects
  SchEq (ANR-12-JS01-0005-01) and BECASIM
  (ANR-12-MONU-0007-04).} 
\maketitle

\section{Introduction}
\label{sec:intro}

We consider the large time behavior for the nonlinear Schr\"odinger
equation
\begin{equation}
  \label{eq:nls}
  i\d_t u +\frac{1}{2}\Delta u = V(x)u + |u|^{2\si}u,
\end{equation}
where $u:(t,x,y)\in \R\times \R\times \R^{d-1}\to \C $, with $d\ge 2$,
$\Delta$ is the Laplacian in $(x,y)$,
and $0<\si<\frac{2}{(d-2)_+}$ (where $1/a_+$ stands for $+\infty$ if $a\le 0$, and for $1/a$ if $a>0$): the nonlinearity is energy-subcritical
in terms of the whole space dimension $d$. The external potential $V$
depends only on $x$. More precisely, we suppose:
\begin{hyp}
\label{hyp:V} 
The potential $V\in L^2_{\rm loc}(\R)$ is real-valued and bounded from
below: 
\begin{equation*}
  \exists C_0,\quad V(x)+C_0\ge 0,\quad \forall x\in \R. 
\end{equation*}
\end{hyp}
It follows from \cite[Theorem X.28]{ReedSimon2} that 
\begin{equation*}
  H = -\frac{1}{2}\Delta +V(x)
\end{equation*}
is essentially self-adjoint on $C_0^\infty(\R^d)$, with domain
(\cite[Theorem X.32]{ReedSimon2})
\begin{equation*}
  D(H) = \{f\in L^2(\R^d),\quad -\frac{1}{2}\Delta f+Vf\in L^2(\R^d)\}. 
\end{equation*}
The goal of this paper is to understand the large time dynamics in
\eqref{eq:nls}. This framework is to be compared with the analysis in
\cite{TzVi-p}, where there is no external potential ($V=0$), but where
the $x$ variable belongs to the torus $\T$ (which is the only
one-dimensional compact manifold without boundary).  It is proven
there that if a short range scattering theory is available for the
nonlinearity $|u|^{2\si}u$ in $H^1(\R^{d-1})$, that is if
$\frac{2}{d-1}<\si<\frac{2}{(d-2)_+}$, then the solution of the Cauchy
problem for $(x,y)\in
\T\times \R^{d-1}$ (is global and) is asymptotically linear as $t\to
\infty$.
\smallbreak

 In this paper, we prove the analogous result in the case of
\eqref{eq:nls}, as well as the existence of wave operators (Cauchy
problem with behavior prescribed at infinite time). This extends some
of the results from \cite{AnCaSi-p} where the special case of an
harmonic potential $V$ is considered. The  properties related
to the harmonic potentials are exploited to prove the existence of
wave operators in the case of a multidimensional confinement
($V(x)=|x|^2$, $x\in \R^n$, $n\ge 1$), a case that we do not consider
in the present paper (see Remark~\ref{rem:higherD}): essentially, if the nonlinearity is
short range on $\R^{d-n}$, then it remains short range on $\R^d$ with
$n$ confined directions. Long range effects are described in
\cite{HaTh-p}, in the case $n=d-1$ and $\si=1$ (cubic nonlinearity,
which is exactly the threshold to have long range scattering in one
dimension). A technical difference with \cite{TzVi-p} is that for the
Cauchy problem, we do
not make use of inhomogeneous Strichartz for non-admissible pairs like
established in \cite{CW92,FoschiStri,Vil07}, and for scattering
theory, such estimates are not needed when $d\le 4$. 
\smallbreak

We emphasize that here,  the potential $V$ 
 can have essentially any behavior, provided that it remains
bounded from below. It can be bounded (in which case the term
``confinement'' is inadequate), or grow arbitrarily fast as $x\to \pm
\infty$. This is in sharp contrast with
e.g. \cite{Miz14,YajZha01,YajZha04}, where Strichartz estimates (with
loss) are established in the presence of super-quadratic potentials,
or with \cite{BCM08}, where a functional calculus adapted to confining
potentials is developed: in all these cases, typically, an exponential growth of
the potential is ruled out, since in this case, no pseudo-differential
calculus is available. 
\smallbreak

Introduce the notation
\begin{equation*}
  M_x =-\frac{1}{2}\d_x^2+V(x)+C_0.
\end{equation*}
We define the spaces 
\begin{align*}
  & B_x=\left\{u\in L^2(\R),M_x^{1/2}u\in L^2(\R)\right\},\quad
\Sigma_y=\left\{u\in H^1(\R^{d-1}),yu\in L^2(\R^{d-1})\right\}, \\
&  Z = L^2_yB_x\cap L^2_xH^1_y,\quad \tilde Z=L^2_yB_x\cap L^2_x\Sigma_y,
\end{align*}
endowed with the norms
$$\|u\|_{B_x}^2=\|u\|_{L^2_x(\R)}^2+\|M_x^{1/2}u\|_{L^2_x(\R)}^2=
\|u\|_{L^2_x(\R)}^2+\<M_x u,u\>,$$
$$\|u\|_{\Sigma_y}^2=\|u\|_{L^2_y(\R^{d-1})}^2+\|\nabla_y
  u\|_{L^2_y(\R^{d-1})}^2+\|yu\|_{L^2_y(\R^{d-1})}^2,$$ 
and
$$\|u\|_Z^2=\|u\|_{L^2_{xy}(\R^d)}^2+\|M_x^{1/2}u\|_{L^2_{xy}(\R^d)}^2+
  \|\nabla_y u\|_{L^2_{xy}(\R^d)}^2 ,\quad \|u\|_{\tilde Z}^2=
  \|u\|_Z^2+ \|y u\|_{L^2_{xy}(\R^d)}^2.$$  
The group $e^{-itH}$ is unitary on $Z$, but not on $\tilde Z$, a
property which is discussed  in the proof of
Lemma~\ref{lem:opA}. 
  \begin{remark}
    Note that $B_x$ is the domain of the operator $M_x^{1/2}$, which is
defined as a fractional power of the self-adjoint operator $M_x$
acting on $L^2(\R)$: for $u\in B_x$, $M_x^{1/2}u$ is defined by

$$M_x^{1/2}u=\int_0^\infty\lambda^{1/2}dE_\lambda(u),$$
where $M_x=\int_0^\infty\lambda dE_\lambda$ is the spectral
decomposition of $M_x$.
\end{remark}

\begin{theorem}[Cauchy problem]\label{theo:cauchy}
 Let $d\ge 2$, $V$ satisfying Assumption~\ref{hyp:V} and
 $0<\si<\frac{2}{(d-2)_+}$.  Let $t_0\in \R$ and $u_0\in Z$. There
 exists a unique solution 
 $  u\in C(\R;Z) $ to \eqref{eq:nls} such that $u_{\mid
   t=t_0}=u_0$. The following two quantities are independent of time:
 \begin{align*}
   &\text{Mass: }\|u(t)\|_{L^2_{xy}(\R^d)}^2,\\
&\text{Energy: }\frac{1}{2}\|\nabla_{xy}u(t)\|_{L^2_{xy}(\R^d)}^2 +
  \frac{1}{\si+1}\|u(t)\|_{L_{xy}^{2\si+2}(\R^d)}^{2\si+2} +\int_{\R^d} V(x) |u(t,x,y)|^2dxdy.
 \end{align*}
If in addition $u_0\in \tilde Z$, then $u\in
 C(\R;\tilde Z)$.\end{theorem}

\begin{theorem}[Existence of wave operators]\label{theo:waveop}
   Let $d\ge 2$, and $V$ satisfying Assumption~\ref{hyp:V}. \\
$1.$ If $u_-\in Z$ and $\frac{2}{d-1}\le \si< \frac{2}{(d-2)_+}$,
there exists $u\in C(\R;Z)$
solution to \eqref{eq:nls} such that
\begin{equation*}
  \| u(t)-e^{-itH}u_-\|_Z=\|e^{itH} u(t)-u_-\|_Z\Tend t {-\infty}0.
\end{equation*}
This solution is such that
\begin{equation*}
 u\in L^\infty(\R;Z) \cap L^p((-\infty,0];L^k_yL^2_x)
\end{equation*}
for some pair $(p,k)$ given in the proof, and it is unique in this class. \\
$2.$ If $u_-\in \tilde Z$ and
$\frac{2}{d}<\si<\frac{2}{(d-2)_+}$,
there exists a unique $u\in C(\R;\tilde Z)$
solution to \eqref{eq:nls} such that
\begin{equation*}
 e^{itH} u\in L^\infty(]-\infty,0];\tilde Z) \quad\text{and}\quad
 \|e^{itH} u(t)-u_-\|_{\tilde Z}\Tend t {-\infty}0. 
\end{equation*}
\end{theorem}
  In the second case, the lower bound $\si>\frac{2}{d}$ is weaker than
  in the first case, so there is some gain in working in the smaller
  space $\tilde Z$ rather than in $Z$. However, this lower bound is larger than in the
  corresponding 
  result from \cite{AnCaSi-p} where only the case $V(x)=x^2$ is
  considered. Indeed in \cite{AnCaSi-p}, the general lower bound is
  $\si>\frac{2d}{d+2}\frac{1}{d-1}$, which is smaller than the present
  one as soon as $d\ge 3$. The main technical reason is that specific
  properties of the harmonic oscillator (typically, the fact that it
  generates a flow which is periodic in time) makes it possible to
  establish a larger set of Strichartz estimates than the one which we
  use in the present paper. In all cases, the expected borderline
  between short range and long range scattering is
  $\si_c=\frac{1}{d-1}$ ($d-1$ is the ``scattering dimension''), so
  our result is sharp in the case $d=2$, and 
  most likely only in this case.

\begin{theorem}[Asymptotic completeness]\label{theo:AC}
   Let $d\ge 2$, $V$ satisfying Assumption~\ref{hyp:V}, and
   $\frac{2}{d-1}<\si<\frac{2}{(d-2)+}$. For any $u_0\in Z$, there
   exists a unique $u_+\in Z$ such that the solution to \eqref{eq:nls}
   with $u_{\mid t=0}=u_0$ satisfies
   \begin{equation*}
     \|u(t)-e^{-itH}u_+\|_Z =\|e^{itH}u(t)-u_+\|_Z\Tend t {+\infty}
     0. 
   \end{equation*}
\end{theorem}

\begin{remark}\label{rem:higherD}
  When a confinement is present (due either to a harmonic potential,
  or  to a bounded geometry) in $n$ directions, for a total space
  dimension $d$, it is expected that the ``scattering dimension'' is
  $d-n$. This was proven systematically in the case of a harmonic
  confinement in \cite{AnCaSi-p}, complemented by
  \cite{HaTh-p}; see also \cite{HPTV-p,TzVi12}. Therefore, to prove asymptotic 
  completeness thanks to Morawetz estimates, it is natural to assume
  $\si>\frac{2}{d-n}$ (essentially because it is not known how to take
  advantage of these estimates otherwise, except in the $L^2$-critical
  case, where many other tools are used). On the other hand, for the Cauchy
  problem to be locally well-posed at the $H^1$-level, it is necessary
  to assume $\si \le \frac{2}{d-2}$ if $d\ge 3$. For the above two
  conditions to be consistent in the energy-subcritical case
  $\si<\frac{2}{d-2}$, we readily see that the only possibility is
  $n=1$, as in \cite{TzVi-p} and the present paper. To treat the case
  $n=2$,the analysis of a doubly critical case would be required: $L^2$-critical
  in $\R^{d-n}$ with $\si=\frac{2}{d-n}$, and energy-critical in $\R^d$
  with $\si=\frac{2}{d-2}$. 
\end{remark}

\section{Technical preliminaries}
\label{sec:prelim}

\subsection{Sobolev embeddings}

\begin{lemma}\label{lem-emb-bm-hm}
$B_x$ is continuously embedded into $H^{1}_x(\R)$. 
\end{lemma}

\begin{proof}
Since $V $ is bounded from below, we have
\begin{align*}
  \|u\|_{H^1_x(\R)}^2&\le\|u\|_{L^2_x(\R)}^2+\|\partial_xu\|_{L^2_x(\R)}^2+2\int_\R
\(V(x)+C_0\)|u(x)|^2dx\\
&\le \|u\|_{L^2_x(\R)}^2+2\<M_x u,u\>\lesssim\|u\|_{B_x}^2,
\end{align*}
hence the result. 
\end{proof}
Introduce, for $\gamma,s\ge 0$,  the anisotropic Sobolev space
$$H^\gamma_yH^s_x=(1-\Delta_y)^{-\gamma/2}(1-\partial_x^2)^{-s/2}L^2_{x,y},$$
endowed with the norm
$$\|u\|_{H^\gamma_yH^s_x}^2=\int_{\R\times\R^{d-1}}\<\xi\>^{2s}
\<\eta\>^{2\gamma}|\widehat{u}(\xi,\eta)|^2d\xi 
d\eta,$$
where $\hat{u}$ denotes the Fourier transform of $u$ in both $x$ and
$y$ variables. $\dot{H}^\gamma_yH^s_x$ denotes the corresponding
homogeneous space, endowed with the norm
$$\|u\|_{\dot{H}^\gamma_yH^s_x}^2=\int_{\R\times\R^{d-1}}\<\xi\>^{2s}|\eta|^{2\gamma}|\widehat{u}(\xi,\eta)|^2d\xi
d\eta.$$
\begin{lemma}\label{lem2}
If $\eps\in (0,1/2)$, $s=\frac{1}{2}+\eps$ and
$\gamma=\frac{1}{2}-\eps$, then 
$$\|u\|_{\dot{H}^{\gamma}_yH^{s}_x}\le
\|u\|_{{H}^{\gamma}_yH^{s}_x}\lesssim\|u\|_{Z},\quad \forall 
u\in Z.$$
\end{lemma}

\begin{proof}
 From Young inequality and Lemma~\ref{lem-emb-bm-hm}, 
\begin{align*}
\|u\|_{{H}^{\gamma}_yH^{s}_x}^2&=\int_{\R\times\R^{d-1}}\<\xi\>^{2\gamma}
\<\eta\>^{2s}|\widehat{u}(\xi,\eta)|^2d\xi 
d\eta\\ 
& \lesssim 
\int_{\R\times\R^{d-1}}\left[(1+\xi^2)+(1+|\eta|^2)\right]|\widehat{u}(\xi,\eta)|^2d\xi
d\eta\ \lesssim\ \|u\|_{L^2_yH^1_x}^2+\|u\|_{L^2_x\dot{H}^1_y}^2,
\end{align*}
hence the result. 
\end{proof}
\subsection{Anisotropic Gagliardo-Nirenberg inequality}

\begin{proposition}\label{sob-anis}
Let $k,s,\gamma>0$  such that
\begin{equation}\label{ass-ksg}
s>1/2\quad \text{and} \quad \frac{1}{2}>
\frac{1}{k}>\frac{1}{2}-\frac{\gamma}{d-1}>0.
\end{equation}
Then $H^\gamma_yH^s_x\subset L^k_yL^\infty_x$, and there exists $C>0$
such that for every 
$u\in H^\gamma_yH^s_x$,
\begin{equation*}
  \|u\|_{L^k_yL^\infty_x}\le C
  \|u\|_{L^2_yH^{s}_x}^{1-\delta}\|u\|_{\dot{H}^{\gamma}_yH^{s}_x}^{\delta},
  \quad\text{where }\delta=\frac{d-1}{\gamma
  }\left(\frac{1}{2}-\frac{1}{k}\right). 
\end{equation*}
\end{proposition}

\begin{proof}
We first use the Sobolev inequality in the $x$ variable and Minkowski
inequality (which is possible because $k>2$). We get
\begin{equation}\label{23}
\|u\|_{L^k_yL^\infty_x}\lesssim
\|u\|_{L^k_yH^{s}_x}=\|\<\xi\>^{s}\mathcal{F}_xu(\xi,y)\|_{L^k_yL^2_\xi}
\lesssim\|\<\xi\>^s\mathcal{F}_xu(\xi,y)\|_{L^2_\xi   L^k_y},
\end{equation}
where $\mathcal{F}_x$ denotes the
Fourier transform in the $x$ variable. Similarly, we denote by $\mathcal{F}_y$  the
Fourier transform in $y$ and
$\widehat{u}(\xi,\eta)=(\mathcal{F}_x\mathcal{F}_yu)(\xi,\eta)$. Then
for a fixed value of $\xi\in\R$, Hausdorff-Young 
inequality yields
\begin{equation}\label{24}
\|\mathcal{F}_xu(\xi,y)\|_{L^k_y}\lesssim \|\widehat{u}(\xi,\eta)\|_{L^{k'}_\eta}.
\end{equation}
Omitting the dependence of the right hand side in $\xi$, let us denote by
$v(\eta)=\widehat{u}(\xi,\eta)$. It follows from
the triangle and H\"older inequality that for any $R>0$,
\begin{align}
\|v\|_{L^{k'}_\eta} &\le
\|v\|_{L^{k'}(|\eta|<R)}+\|v\|_{L^{k'}(|\eta|>R)}\nonumber\\
&\lesssim 
\|\mathbf{1}_{\{|\eta|<R\}}\|_{L^p(|\eta|<R)}\|v\|_{L^2_\eta}+\||\eta|^{-\gamma}\|_{L^p(|\eta|>R)}\||\eta|^{\gamma}v\|_{L^2_\eta}\nonumber\\
&\lesssim 
R^{(d-1)/p}\|v\|_{L^2_\eta}+R^{(d-1)/p-\gamma}\||\eta|^{\gamma}v\|_{L^2_\eta},\label{25}
\end{align}
where $p$ is given by $1/p=1/2-1/k$.. Note that (\ref{ass-ksg}) implies
that $\gamma p>d-1$, and 
therefore $|\eta|^{-\gamma}\in L^p(|\eta|>R)$. Optimizing in $R$ in the right hand side of \eqref{25}, we get
\begin{equation}\label{26}
\|v\|_{L^{k'}_\eta} \lesssim \|v\|_{L^2_\eta}^{1-\delta}\||\eta|^{\gamma}v\|_{L^2_\eta}^\delta,
\end{equation}
where $\delta=\frac{d-1}{\gamma p}\in (0,1)$. Combining \eqref{23},
\eqref{24} and \eqref{26}, H\"older inequality yields
\begin{align*}
\|u\|_{L^k_yL^\infty_x}&\lesssim\left(
\int \<\xi\>^{2s(1-\delta)}\|\widehat{u}\|_{L^2_\eta}^{2(1-\delta)}
\<\xi\>^{2s\delta}\||\eta|^{\gamma}\widehat{u}\|_{L^2_\eta}^{2\delta}d\xi\right)^{1/2}\\ 
&\lesssim \left(
\int
  \<\xi\>^{2s}\|\widehat{u}\|_{L^2_\eta}^{2}d\xi\right)^{(1-\delta)/2}
\left(\int \<\xi\>^{2s}\||\eta|^{\gamma}\widehat{u}\|_{L^2_\eta}^{2}
d\xi\right)^{\delta/2}\\
&=\|u\|_{L^2_yH^{s}_x}^{1-\delta}\|u\|_{\dot{H}^{\gamma}_yH^{s}_x}^{\delta}.
\end{align*}
\end{proof}

\begin{corollary}\label{coro}
Let $2<k<\frac{2(d-1)}{(d-2)_+}$. Then $Z$ is continuously embedded in $L^k_yL^\infty_x$.
\end{corollary}

\begin{proof}
Pick $\eps>0$ small enough such that
$$\frac{1}{2}-\frac{1/2-\eps}{d-1}=\frac{d-2}{2(d-1)}+\frac{\eps}{d-1}<\frac{1}{k}.$$
Then $(s,\gamma)=(1/2+\eps,1/2-\eps)$ satisfy the assumptions of Proposition  \ref{sob-anis} and Lemma \ref{lem2}. Thus, using also Lemma \ref{lem-emb-bm-hm}, 
\begin{equation}
\|u\|_{L^k_y L^\infty_x}\lesssim
\|u\|_{L^2_yB_x}^{1-\delta}
\|u\|_Z^{\delta}\lesssim\|u\|_Z.\nonumber
\end{equation}
\end{proof}

\subsection{Strichartz estimates}

Following the idea from \cite{TzVi12}, with the generalization from
\cite{AnCaSi-p} (noticing that the spectral decomposition from the
proof in \cite{TzVi12} is not needed), we have, since $M_x$ commutes
with $H$:
\begin{proposition}\label{prop:strichartz}
Let $d\ge 2$. We have
\begin{equation*}
\|e^{-itH}u_0\|_{L^q_tL^r_yL^2_x}+
\left\|\int_0^te^{-i(t-s)H}F(s)ds\right\|_{L^{q_1}_tL^{r_1}_yL^2_x}\lesssim 
\|u_0\|_{L^2_{y}L^2_x}+
\|F\|_{L^{q_2'}_tL^{r_2'}_yL^2_x},
\end{equation*}
provided that the pairs are $(d-1)$-admissible, that is
\begin{equation*}
  \frac{2}{q}+\frac{d-1}{r}=  \frac{2}{q_1}+\frac{d-1}{r_1}=
  \frac{2}{q_2}+\frac{d-1}{r_2}=\frac{d-1}{2}, 
\end{equation*}
 with $(q,r)\neq (2,\infty)$
if $d=3$.
\end{proposition}

\subsection{Vectorfields}

We introduce the notation
\begin{align*}
   &A_0(t)=A_0={\rm Id}, \quad A_1(t)=A_1=M^{1/2}_x, \quad  A_2(t)=A_2=\nabla_y,\\
  &
A_3(t)=y+it\nabla_y=it e^{i|y|^2/(2t)}\nabla_y\(\cdot \ e^{-i|y|^2/(2t)}\)=e^{-itH}ye^{itH}. 
\end{align*}
The operator $A_3$ is the standard Galilean operator on $\R^{d-1}$,
see e.g. \cite{CazCourant}, so the last identity stems from the fact
that $e^{-itM_x}$ commutes with both $e^{i\frac{t}{2}\Delta_y}$ and
$y$. We readily have:
\begin{lemma}\label{lem:opA}
  The operators $A_j$ satisfy the following properties:
  \begin{itemize}
  \item Commutation: for $j\in \{0,\dots,3\}$, $[i\partial_t-H,A_j]=0$.
\item Action on the nonlinearity: for all $j\in \{0,\dots,3\}$,
  \begin{equation*}
    \left\|A_j\(|u|^{2\si}u\)\right\|_{L^2_x}\lesssim
    \|u\|_{L^\infty_x}^{2\si}\|A_ju\|_{L^2_x}. 
  \end{equation*}
\item Equivalence of norms: for all $u\in C_0^\infty(\R^d)$, we have,
  uniformly in $t\in \R$, 
  \begin{equation}\label{eq:equiv-norm}
    \| e^{itH} u\|_{Z}= \|u\|_Z\approx
 \sum_{j=0}^2\|A_ju\|_{L^2_{xy}},\quad \| e^{itH} u\|_{\tilde Z}\approx
 \sum_{j=0}^3\|A_j(t)u\|_{L^2_{xy}} .
  \end{equation}
\item Gagliardo-Nirenberg inequalities: for all $g\in
  \Sigma_y$, $2\le p< \frac{2}{(d-3)_+}$,
  \begin{align*}
    &\| g\|_{L^p_y}\le C \|g\|_{L^2}^{1-\delta}\|A_2
      g\|_{L^2_y}^\delta,\\
&\| g\|_{L^p_y}\le \frac{C}{|t|^{\delta}} \|g\|_{L^2}^{1-\delta}\|A_3(t)
      g\|_{L^2_y}^\delta,\quad t\not =0,
  \end{align*}
where $C$ is independent of $t$, and
$\delta=(d-1)\(\frac{1}{2}-\frac{1}{p}\)$. 
  \end{itemize}
\end{lemma}
\begin{proof}
  The commutation property is straightforward. For the action on the
  nonlinearity, it is trivial in the case of $A_0$ and $A_2$. For
  $A_3$, it stems classically from the fact that $A_3$ is the gradient in $y$
  conjugated by an exponential of modulus one and that the
  nonlinearity we consider is gauge invariant. Concerning $A_1$, we
  compute
  \begin{align*}
    \|M_x^{1/2}\left(|u|^{2\sigma}u\right)\|_{L^2_x}^2 & = \<M_x
                                                         \(|u|^{2\si}u\),|u|^{2\si}u\>\\
&=\frac{1}{2}\|\partial_x\left(|u|^{2\sigma}u\right)\|_{L^2_x}^2+\int_{-\infty}^{+\infty}\(V(x)+C_0\)|u|^{4\sigma+2}dx\\
&\le 
(2\sigma+1)^2\|u\|_{L^\infty_x}^{4\sigma}\left(\frac{1}{2}\|\partial_x
  u\|_{L^2_x}^2+\int_{-\infty}^{+\infty}\(V(x)+C_0\)|u|^{2}dx\right)\\
& = (2\sigma+1)^2\|u\|_{L^\infty_x}^{4\sigma}\|M_x^{1/2}u\|_{L^2_x}^2.
  \end{align*}
Recall that $A_0,A_1$ and $A_2$
commute with $e^{itH}$, which is unitary on
$L^2(\R^d)$, hence the first equivalence of norms. The identity 
$A_3(t) = e^{-itH}y e^{itH}$ yields the second equivalence of norms,
uniformly in time: note that $\|e^{itH}u\|_{\tilde Z}$ is equivalent
to $\|u\|_{\tilde Z}$ only locally in time, due to the factor $t$ in
the identity $A_3(t) =y+it\nabla_y$. 
\smallbreak

Finally, the Gagliardo-Nirenberg inequalities stated in the lemma are
the classical ones, using once more the factorization of $A_3$. 
\end{proof}

\section{Cauchy problem}
\label{sec:cauchy}

In this section, we prove Theorem~\ref{theo:cauchy}. The existence
part relies on a a standard fixed point argument, adapted to the
present framework. Since the problem is invariant by translation in
time, we may assume $t_0=0$. Duhamel's formula reads
\begin{equation*}
  u(t)=e^{-itH}u_0-i\int_{0}^te^{-i(t-s)H}\left(|u|^{2\sigma}u\right)(s)ds=:\Phi(u)(t).
\end{equation*}
This Cauchy problem will be solved thanks to a fixed point argument in
a ball of the Banach space
\begin{equation*}
  Z_T=\{u\in L^\infty([0,T];Z),\quad A_ju\in
  L^q\([0,T];L^r_yL^2_x\),
  \forall j\in \{0,1,2\}\},
\end{equation*}
where $(q,r)$ is a $(d-1)$-admissible
pair that will be fixed later. The space $Z_T$ is naturally equipped
with the norm
\begin{align*}
  \|u\|_{Z_T}= \sum_{j=0}^2 \(\|A_ju\|_{L^\infty_T
  L^2_{xy}}+\|A_ju\|_{L^q_TL^r_yL^2_x}\) . 
\end{align*}
Denote
$L^a_TX=L^a([0,T];X)$. Proposition~\ref{prop:strichartz} and 
the first point of Lemma~\ref{lem:opA} imply, for $ j\in \{0,1,2\}$:
\begin{equation*}
  \|A_j\Phi(u)\|_{L^\infty_TL^2_{xy}} + \|A_j\Phi(u)\|_{L^q_TL^r_yL^2_x}\lesssim \|A_j
u_0\|_{L^2_{xy}}+\|A_j(|u|^{2\sigma}u)\|_{L^{q'}_TL^{r'}_yL^2_x}.
\end{equation*}
The second point of  Lemma~\ref{lem:opA} and H\"older inequality yield
\begin{equation*}
 \|A_j(|u|^{2\sigma}u)\|_{L^{q'}_TL^{r'}_yL^2_x} \lesssim
 \|u\|_{L^\theta_TL^k_yL^\infty_x}^{2\sigma}\|A_j 
u\|_{L^{q}_TL^{r}_yL^2_x},
\end{equation*}
where $\theta$ and $k$ are given by
\begin{equation}\label{def-theta-k}
\frac{1}{q'}=\frac{2\sigma}{\theta}+\frac{1}{q},\quad
\frac{1}{r'}=\frac{2\sigma}{k}+\frac{1}{r}.
\end{equation}
We infer
\begin{equation}
\label{phi0ZT}
\|\Phi(u)\|_{Z_T}\lesssim
\|u_0\|_Z+\|u\|_{L^\theta_TL^k_yL^\infty_x}^{2\sigma}\|u\|_{Z_T}. 
\end{equation}
 Let us now explain how the parameters $q,r,\theta,k$ are
chosen.
\smallbreak

\noindent {\bf Case $d=2$.} We choose $r\in(2,\infty)$ if $\sigma\ge
1$, $2<r<\frac{2}{1-\si} $ if $0<\sigma<1$, and $(q,r)$ the corresponding
$1$-admissible pair. Then, \eqref{def-theta-k} defines a number $k$
that 
 belongs to $(2,\infty)$
.
\smallbreak

\noindent {\bf Case $d=3$.} $(q,r)$ is a
$2$-admissible pair with $r\in (2,\infty)$ such
that 
$$\frac{1}{4}<\frac{1}{2\sigma}\left(1-\frac{2}{r}\right)=:\frac{1}{k}<\frac{1}{2}.$$
Note that this is made possible thanks to the assumption
$\sigma<2$. 
\smallbreak

\noindent {\bf Case $d\ge 4$.} As $(q,r)$ describes the set of all
$(d-1)$-admissible pairs, 
$r$ varies between the two extremal values $2$ and $\frac{2(d-1)}{d-3}$, and
therefore $\frac{1}{2\sigma}(1-\frac{2}{r})$ varies between $0$ and
$\frac{1}{\sigma(d-1)}$, where the latter number is larger than
$\frac{d-2}{2(d-1)}$ thanks  to the assumption 
$\sigma<2/(d-2)$. Thus, one can choose $2<r<\frac{2(d-1)}{d-3}$ such
that if $k$ is defined by 
\eqref{def-theta-k}, 
$$\frac{d-2}{2(d-1)}<\frac{1}{k}<\frac{1}{2}.$$ 

For these choices of the 
parameters, Corollary \ref{coro} and
H\"older inequality in time imply
\begin{equation}\label{l-theta-k}
\|u\|_{L^\theta_TL^k_yL^\infty_x}\lesssim\|u\|_{L^\theta_TZ}
\lesssim T^{1/\theta}\|u\|_{Z_T}.
\end{equation}
Note that we have chosen admissible pairs such that $q>2$. Thus, since
$\theta$ is defined by \eqref{def-theta-k}, $1/\theta>0$. From the
combination of \eqref{phi0ZT} and \eqref{l-theta-k}, we deduce that if
$u$ belongs to the ball $B(R,Z_T)$ of $Z_T$ with radius $R>0$ centered at the
origin, we have
\begin{equation}\label{stab_ball}
\|\Phi(u)\|_{Z_T}\le
C_1\|u_0\|_Z+CT^{2\sigma/\theta}R^{2\sigma+1}.
\end{equation}
Chosing $R=2C_1\|u_0\|_Z$ and $T=T(\|u_0\|_Z)>0$ sufficiently
small, $B(R,Z_T)$ is 
stable by $\Phi$. Then, we note that $B(R,Z_T)$ endowed with the
norm 
$$\|u\|_{B(R,Z_T)}=\|u\|_{L^\infty_TL^2_{xy}}+
\|u\|_{L^q_TL^r_yL^2_x}$$ 
is a complete metric space (Kato's method, see
e.g. \cite{CazCourant}). 
For $u_2,u_1\in B(R,Z_T)$, the same estimates as above yield
\begin{align*}
  \|\Phi(u_2) -\Phi( u_1)\|_{L^\infty_T L^2_{xy}}+&\|\Phi(u_2) - \Phi(u_1)\|_{L^q_T
                                        L^r_yL^2_x}\\
&\lesssim
  \(\|u_2\|_{L^\theta_T L^k_yL^\infty_x}^{2\si}+
  \|u_1\|_{L^\theta_TL^k_yL^\infty_x}^{2\si}\)\|u_2-u_1\|_{L^q_T
  L^r_yL^2_x}\\
&\lesssim T^{2\si/\theta} \( \|u_2\|_{Z_T}^{2\si} +
  \|u_1\|_{Z_T}^{2\si} \)
\|u_2-u_1\|_{L^q_T
  L^r_yL^2_x} \\
&\lesssim T^{2\si/\theta} R^{2\si}
\|u_2-u_1\|_{L^q_T
  L^r_yL^2_x} .
\end{align*}
Therefore, $\Phi$ is a
contraction on $B(R,Z_T)$ endowed with the above norm, provided that 
$T=T(\|u_0\|_Z)$ is 
sufficiently small, hence the existence of a local solution in $Z$. 
\smallbreak

The conservation of mass and energy follows from standard arguments
(see e.g. \cite{CazCourant}). Under Assumption~\ref{hyp:V}, this
implies an a priori bound for $\|u(t)\|_Z$, and so the solution $u$ is
global in time, $u\in L^\infty(\R;Z)$. 
\smallbreak

Unconditional uniqueness as stated in Theorem~\ref{theo:cauchy}
follows from the same approach as in \cite{TzVi-p}. If $u_1,u_2\in
C([0,T];Z)$ are two solutions of \eqref{eq:nls} with the same initial
datum, then
\begin{equation*}
  u_2(t)-u_1(t) = -i\int_0^t e^{-i(t-s)H}
  \(|u_2|^{2\si}u_2-|u_1|^{2\si}u_1\)(s)ds. 
\end{equation*}
Resuming the same estimates as above, we now have, for
$0<\tau\le T$:
\begin{align*}
  \|u_2 - u_1\|_{L^q_\tau L^r_yL^2_x}&\lesssim
  \(\|u_2\|_{L^\theta_\tau L^k_yL^\infty_x}^{2\si}+
  \|u_1\|_{L^\theta_\tau
                                       L^k_yL^\infty_x}^{2\si}\)\|u_2-u_1\|_{L^q_\tau L^r_yL^2_x}\\
&\lesssim  \tau^{2\si/\theta} \( \|u_2\|_{Z_T}^{2\si} +
  \|u_1\|_{Z_T}^{2\si} \)
\|u_2-u_1\|_{L^q_\tau
  L^r_yL^2_x},
\end{align*}
and uniqueness follows by taking $\tau>0$ sufficiently small. 
\smallbreak

To complete the proof of Theorem~\ref{theo:cauchy}, we just have to
check that the extra regularity $u_0\in \tilde Z$ is propagated by the
flow.  To do so, it suffices to replace the space $Z_T$
 with
\begin{equation*}
  \tilde Z_T=\{u\in L^\infty((0,T),Z),\quad A_j(t)u\in
  L^q\((0,T);L^r_yL^2_x\),
  \forall j\in \{0,1,2,3\}\},
\end{equation*}
that is, to add the field $A_3$. The second point of
Lemma~\ref{lem:opA},  and the above computations then yield
\begin{align*}
  \|A_3\Phi(u)\|_{L^\infty_TL^2_{xy}}+
  \|A_3\Phi(u)\|_{L^q_TL^r_yL^2_{x}}&\lesssim \|yu_0\|_{L^2_{xy}}+
  \|u\|_{L^\theta_T
  L^k_yL^\infty_x}^{2\si}\|A_3u\|_{L^q_TL^r_yL^2_{x}}\\
&\lesssim \|yu_0\|_{L^2_{xy}}+
  T^{2\si/\theta}\|u\|_{Z_T}^{2\si}\|A_3u\|_{L^q_TL^r_yL^2_{x}}.
\end{align*}
The above fixed point argument can then be resumed: we construct a
local solution in $\tilde Z$, $u\in C([-T,T];\tilde Z)\cap L^\infty(\R;Z)$. The latest
property and the previous estimate show that $A_3u\in C(\R;L^2_{xy})$
is global in time.

\section{Existence of wave operators}
\label{sec:wave}

To prove the existence of wave operators, we construct a fixed point
for the related Duhamel's formula,
\begin{equation}\label{eq:duhamel-}
  u(t)=e^{-itH}u_--i\int_{-\infty}^t e^{-i(t-s)H}\(|u|^{2\si}u\)(s)ds
  =:\Phi_-(u)(t),
\end{equation}
on some time interval $(-\infty,-T]$ for $T$ possibly very large but
finite. According to the regularity assumption on $u_-$, we construct
a solution in $Z$ or in $\tilde Z$. This solution is actually global
in time from either case of Theorem~\ref{theo:cauchy}. We therefore
focus on the construction of a fixed point for $\Phi_-$, as well as on
uniqueness. In a similar fashion as in Section~\ref{sec:cauchy}, we
denote $L_T^aX = L^a((-\infty,-T];X)$. 

\subsection{Wave operators in $Z$}

 Resume the  $(d-1)$-admissible pair $(q,r)$ used in
 Section~\ref{sec:cauchy}, and
 $(\theta,k)$  given by \eqref{def-theta-k}. 
 For $(q_1,r_1)$ a $(d-1)$-admissible pair,  and $j\in
 \{0,1,2\}$, Strichartz estimates and H\"older 
inequality yield:
\begin{align*}
  \left\| A_j\Phi_-(u)\right\|_{L^{q_1}_T L^{r_1}_yL^2_x}
  &\lesssim\|A_ju_-\|_{L^2_{xy}} + 
\left\| A_j\(|u|^{2\si}u\)\right\|_{L^{q'}_TL^{r'}_y L^2_x}
  \\
 &\lesssim\|A_ju_-\|_{L^2_{xy}} + 
\|u\|_{L^{\theta}_TL^{k}_y L^\infty_x}^{2\si}\| A_j u\|_{L^{q}_TL^{r}_y L^2_x}.
\end{align*}
By construction,
\begin{equation*}
  2\le k<\frac{2(d-1)}{(d-2)_+}< \frac{2(d-1)}{(d-3)_+},
\end{equation*}
so we can find $p$ such that $(p,k)$ is $(d-1)$-admissible. Putting
the definition of admissible pairs and \eqref{def-theta-k} together,
we get
\begin{equation*}
  1-\frac{2\si}{\theta} =\frac{2}{q}= (d-1)\(\frac{1}{2}-\frac{1}{r}\)
  = \frac{(d-1)\si}{k} = \si\(\frac{d-1}{2}-\frac{2}{p}\).
\end{equation*}
By assumption, $\si\ge \frac{2}{d-1}$, so $p\le \theta$, and there
exists $\beta\in (0,1]$ such that
\begin{equation*}
  \|u\|_{L^{\theta}_TL^{k}_y L^\infty_x}\le \|u\|_{L^{p}_TL^{k}_y
    L^\infty_x}^\beta
\|u\|_{L^{\infty}_TL^{k}_y L^\infty_x}^{1-\beta}. 
\end{equation*}
Corollary~\ref{coro} implies
\begin{equation*}
  \left\| A_j\Phi_-(u)\right\|_{L^{q_1}_T L^{r_1}_yL^2_x}
 \lesssim\|A_ju_-\|_{L^2_{xy}} + 
\|u\|_{L^{p}_TL^{k}_y
    L^\infty_x}^{2\si\beta}
\|u\|_{L^{\infty}_TZ}^{2\si(1-\beta)}\| A_j u\|_{L^{q}_TL^{r}_y L^2_x}.
\end{equation*}
Now the one-dimensional Gagliardo-Nirenberg inequality
\begin{equation*}
  \|f\|_{L^\infty_x}\le \sqrt 2 \|f\|_{L^2_x}^{1/2}\|\d_x f\|_{L^2_x}^{1/2}
\end{equation*}
and according to the proof of Lemma~\ref{lem-emb-bm-hm}, we have
\begin{equation}\label{eq:waveH1}
\begin{aligned}
  \left\| A_j\Phi_-(u)\right\|_{L^{q_1}_T L^{r_1}_yL^2_x}
 &\le C\|A_ju_-\|_{L^2_{xy}}  \\
&+
C\|u\|_{L^{p}_TL^{k}_y
    L^2_x}^{\si\beta}\|A_1u\|_{L^{p}_TL^{k}_y
    L^2_x}^{\si\beta}
\|u\|_{L^{\infty}_TZ}^{2\si(1-\beta)}\| A_j u\|_{L^{q}_TL^{r}_y L^2_x}.
\end{aligned}
\end{equation}
for $C$ sufficiently large. We can now define
  \begin{align*}
    B_T:=\Big\{ & u\in C(]-\infty,-T];Z),\\
&\left\|
   A_j u\right\|_{L^q_TL^r_yL^2_x} + \left\|
   A_j u\right\|_{L^\infty_TL^2_{xy}} \le 4C \|A_ju_-\|_{L^2_{xy}},
  \quad j\in \{0,1,2\},\\
& \left\| A_j u\right\|_{L^p_T L^k_yL^2_x} \le 2 \left\|
  A_j  e^{-itH}u_-\right\|_{L^p_T  L^k_yL^2_x} ,\quad j\in \{0,1\}\Big\}.
  \end{align*}
From Strichartz estimates, we know that
for $j\in \{0,1\}$,
\begin{equation*}
 A_j e^{-itH}u_- \in L^p(\R;L^k_yL^2_x),\quad \text{so}\quad
 \left\|A_j e^{-itH}u_-\right\|_{L^p_TL^k_TL^2_x} \to 0\quad \text{as }T\to +\infty.
\end{equation*}
Since $\beta>0$, we infer that $\Phi_-$ maps $B_T$ to itself, for
$T$ sufficiently large, by \eqref{eq:waveH1}, and since the same
estimates yield,
for $j\in \{0,1\}$, 
\begin{align*}
   \left\| A_j\Phi_-(u)\right\|_{L^{p}_T L^{k}_yL^2_x}
 &\le\|A_je^{-itH}u_-\|_{L^{p}_T L^{k}_yL^2_x} \\
&\quad+
C\|u\|_{L^{p}_TL^{k}_y
    L^2_x}^{\si\beta}\|A_1u\|_{L^{p}_TL^{k}_y
    L^2_x}^{\si\beta}
\|u\|_{L^{\infty}_TZ}^{2\si(1-\beta)}\| A_j u\|_{L^{q}_TL^{r}_y L^2_x}.
\end{align*}
We have also, for $u_2,u_1\in B_T$, and typically $(q_1,r_1)\in \{(q,r),(\infty,2)\}$:
\begin{align*}
  \left\| \Phi_-(u_2)-\Phi_-(u_1)\right\|_{L^{q_1}_T L^{r_1}_yL^2_x}&\lesssim
  \max_{j=1,2}\| u_j\|_{L^\theta_TL^k_yL^\infty_x}^{2\si} \left\|
  u_2-u_1\right\|_{L^q_T L^r_yL^2_x}\\
\lesssim \left\|e^{-itH}u_-\right\|_{L^p_TL^k_yL^2_x}^{\si\beta}&
  \left\|A_1 e^{-itH}u_-\right\|_{L^p_TL^k_yL^2_x}^{\si\beta} \|u_-\|_Z^{2\si(1-\beta)} \left\| 
  u_2-u_1\right\|_{L^q_T L^r_yL^2_x}.
\end{align*}
Up to choosing $T$ larger, $\Phi_-$ is a contraction on $B_T$, so $\Phi_-$
has a unique fixed point in $B_T$, which solves
\eqref{eq:duhamel-}. Uniqueness as stated in Theorem~\ref{theo:waveop}
is an easy consequence of the above estimates.

\subsection{Wave operators in $\tilde Z$}
In the case $u_-\in \tilde Z$, we consider the whole set of
vector fields, $(A_j)_{0\le j\le 3}$. For $(q,r)$ a $(d-1)$-admissible
pair to be chosen later, we define
\begin{equation*}
  \tilde Z_T=\{u\in C((-\infty,-T];\tilde Z),\quad A_j(t)u\in
  L^q_TL^r_yL^2_x\cap L^\infty_TL^2_{xy},
  \forall j\in \{0,1,2,3\}\}.
\end{equation*}
We have, for all $(d-1)$-admissible pairs $(q_1,r_1)$, and all
$j\in \{0,1,2,3\}$,
\begin{equation}\label{phi-YT}
  \|A_j\Phi_-(u)\|_{L^{q_1}_TL^{r_1}_yL^2_x}\lesssim \|u_-\|_{\tilde
    Z}+\|u\|_{L^\theta_TL^k_yL^\infty_x}^{2\sigma}\|A_j u\|_{L^q_TL^r_yL^2_x}, 
\end{equation}
where $\theta$ and $k$ are again given by \eqref{def-theta-k}. If
\begin{equation}\label{eq:sob-emb}
H^{1/2-}(\R^{d-1}_y)\hookrightarrow L^k(\R^{d-1}_y),\text{ that is},\quad
   2\le k<\frac{2(d-1)}{d-2},
\end{equation}
we can find  $s$ and $\gamma$ satisfying \eqref{ass-ksg} and $s+\gamma=1$.
To obtain explicit time decay, apply Proposition~\ref{sob-anis} to
$v=e^{-i|y|^2/(2t)}u$. This  yields  
\begin{equation*}
\|u\|_{L^k_yL^\infty_x}=\|v\|_{L^k_yL^\infty_x}\lesssim
\|v\|_{L^k_yH^{s}_x}\lesssim\|v\|_{L^2_yH^{s}_x}^{1-\delta}\|v\|_{\dot{H}^\gamma_yH^{s}_x}^\delta,
\end{equation*}
where $\delta$ is defined by
$$\delta \gamma=(d-1)\left(\frac{1}{2}-\frac{1}{k}\right).$$
Then, since $\gamma+s=1$, it follows from the Young inequality as
in Lemma~\ref{lem2} that
\begin{align}\label{t-g}
\|v\|_{\dot{H}^\gamma_yH^{s}_x}&=|t|^{-\gamma}\left(\int|t\eta|^{2\gamma}(1+\xi^2)^s|\widehat{v}(\xi,\eta)|^2d\xi
d\eta\right)^{1/2}\\
&\lesssim|t|^{-\gamma}\left(\int\left(|t\eta|^{2}+(1+\xi^2)\right)|\widehat{v}(\xi,\eta)|^2d\xi
d\eta\right)^{1/2} \nonumber \\
&\lesssim
  |t|^{-\gamma}\left(\|A_3(t)u\|_{L^2_xL^2_y}+\|u\|_{L^2_yH^1_x}\right),\nonumber 
\end{align}
where in the last line, we have used Plancherel formula and
$$A_3(t)u=it e^{i|y|^2/(2t)}\nabla_y e^{-i|y|^2/(2t)}u=it e^{i|y|^2/(2t)}\nabla_y v.$$
Then, we deduce from \eqref{t-g} and Lemma~\ref{lem2} that for any
$u\in\tilde Z_T$ and $t\le -T$, we have 
\begin{equation}\label{pre-scat}
\|u(t)\|_{L^k_yL^\infty_x} 
\lesssim
\frac{1}{|t|^{(d-1)\left(\frac{1}{2}-\frac{1}{k}\right)}}\sum_{j=0}^3
\|A_j u\|_{L^\infty_TL^2_{xy}}
\lesssim \frac{1}{|t|^{(d-1)\left(\frac{1}{2}-\frac{1}{k}\right)}}\|u\|_{\tilde
  Z_T}.
\end{equation}
Then, provided $t\mapsto |t|^{-(d-1)(1/2-1/k)}$
belongs to $L^\theta(-\infty,-1)$, \eqref{phi-YT} and \eqref{pre-scat}
imply that  for every
$u\in \tilde Z_T$,
\begin{align}\label{phi-YT-bis}
\|A_j\Phi_-(u)\|_{\tilde Z_T}\lesssim\|u_-\|_{\tilde
  Z}+T^{2\sigma\(\frac{1}{\theta}-(d-1)\left(\frac{1}{2}-\frac{1}{k}\right)\)}\|u\|_{\tilde
  Z_T}^{2\sigma+1}.
\end{align}
Let us now explain how the parameters $\theta,k,q,r$ are chosen. Since
$\sigma>1/(d-1)$, one can choose $q>2$ large enough such that 
\begin{equation}\label{cond-scat-0}
(d-1)\sigma>\frac{2}{q}+1.
\end{equation}
Then, $r$ is chosen such that $(q,r)$ is a $(d-1)$-admissible pair, in
such a  way that \eqref{cond-scat-0} becomes
\begin{equation*}\label{cond-scat-1}
(d-1)\left(\sigma+\frac{1}{r}-\frac{1}{2}\right)>1,
\end{equation*}
which is equivalent to 
\begin{equation*}\label{cond-scat-2}
(d-1)\left(\sigma-\frac{2\sigma}{k}\right)=(d-1)\left(\sigma-1+\frac{2}{r}\right)>1-(d-1)\left(\frac{1}{2}-\frac{1}{r}\right)=1-\frac{2}{q}=\frac{2\sigma}{\theta},
\end{equation*}
where $\theta$ and $k$ are defined by \eqref{def-theta-k}. This is
precisely the condition $\theta(d-1)(\frac{1}{2}-\frac{1}{k})>1$ which
ensures that the right hand side of \eqref{pre-scat} belongs to
$L^\theta$. In terms of $k$, \eqref{cond-scat-0} is equivalent to 
\begin{equation*}
  \frac{1}{k}<1-\frac{1}{(d-1)\si}. 
\end{equation*}
This condition is consistent with \eqref{eq:sob-emb} if and only if
\begin{equation*}
  \frac{d-2}{2(d-1)}<1-\frac{1}{(d-1)\si},
\end{equation*}
which is equivalent to $\si>\frac{2}{d}$. 
\smallbreak

The rest of the proof is similar to the proof of local
well-posedness of the Cauchy problem: we take
$R$ and $T$ sufficiently large so  that the ball of radius $R$ in
$\tilde Z_T$ is  stable under the action of
$\Phi_-$, and so that $\Phi_-$ is a contraction on this ball, equipped
with the distance $\|u\|_{L^\infty_TL^2_{xy}}+
\|u\|_{L^q_TL^r_yL^2_x}$, in view of the previous estimates and
\begin{equation*}
  \|\Phi_-(u_2)-\Phi_-(u_1)\|_{L^{q_1}_TL^{r_1}_yL^2_x} \lesssim
  \max_{j=1,2}\|u_j\|^{2\si}_{L^\theta_TL^k_yL^\infty_x}
  \|u_2-u_1\|_{L^q_TL^r_yL^2_x}.
\end{equation*}
In view of \eqref{eq:equiv-norm}, the solution that we have constructed satisfies
\begin{equation*}
  e^{itH} u\in L^\infty((-\infty,-T];\tilde Z).
\end{equation*}
Uniqueness in this class follows from \eqref{eq:equiv-norm} and the
same approach as for the Cauchy problem. If $u_1$ and $u_2$ are two
solutions of \eqref{eq:nls} satisfying
\begin{equation*}
  e^{itH} u_j\in L^\infty((-\infty,-T];\tilde Z),\quad
  \|e^{itH}u_j(t)-u_-\|_{\tilde Z}\Tend t {-\infty}0,\quad j=1,2,
\end{equation*}
then for $\tau>T$,
\begin{equation*}
  \|u_2-u_1\|_{L^{q}_\tau L^{r}_yL^2_x} \lesssim
  \max_{j=1,2}\|u_j\|^{2\si}_{L^\theta_\tau L^k_yL^\infty_x}
  \|u_2-u_1\|_{L^q_\tau L^r_yL^2_x},
\end{equation*}
and \eqref{pre-scat} implies
\begin{equation*}
  \|u_2-u_1\|_{L^{q}_\tau L^{r}_yL^2_x} \lesssim
  \tau^{2\sigma\(\frac{1}{\theta}-(d-1)\left(\frac{1}{2}-\frac{1}{k}\right)\)}
  \|u_2-u_1\|_{L^q_\tau L^r_yL^2_x}.
\end{equation*}
Choosing $\tau$ sufficiently large, we have $u_2=u_1$ for $t\le -\tau$,
and  Theorem~\ref{theo:cauchy} yields $u_2\equiv u_1$.

\section{Asymptotic completeness}
\label{sec:ac}

In this section, we prove Theorem~\ref{theo:AC}. Three approaches are
available  to prove asymptotic completeness for 
nonlinear Schr\"odinger equations (without potential). The initial
approach (\cite{GV79Scatt}) consists in working with a
$\Sigma$ regularity. This  
makes it possible to use the operator $x+it\nabla$, whose main
properties are essentially those stated in Lemma~\ref{lem:opA}, and to
which an important evolution law (the 
pseudo-conformal conservation law) is associated. This law provides
important a priori estimates, from which asymptotic completeness
follows very easily in the case $\si\ge 2/d$, and less easily for
some range of $\si$ below $2/d$; see
e.g. \cite{CazCourant}. Unfortunately, this conservation law seems to
be bound to isotropic frameworks: an analogous identity is available
in the presence on an isotropic quadratic potential (\cite{CaSIMA}),
but in our present framework, anisotropy seems to rule out a similar
algebraic miracle. 
\smallbreak

The second historical approach relaxes the localization assumption on
the data, and allows
to work in $H^1(\R^d)$, provided that $\si>2/d$. It is based on
Morawetz inequalities: asymptotic completeness is then established in
\cite{LiSt78,GV85} for the case $d\ge 3$, and in \cite{NakanishiJFA} for the low
dimension cases $d=1,2$, by introducing more intricate Morawetz
estimates.  
\smallbreak

The most recent approach to prove asymptotic completeness in $H^1$
relies on the introduction of interaction Morawetz estimates in \cite{CKSTTAnnals},
an approach which has been revisited since, in particular in
\cite{PlVe09} and \cite{GiVe10}. In the anisotropic case, interaction
Morawetz have been used in \cite{AnCaSi-p} and \cite{TzVi-p} with two
different angles: in both cases, it starts with the choice of an
anisotropic weight in the virial computation  from
\cite{GiVe10,PlVe09}, but the interpretations of this computation are
then different. We start by presenting a unified statement of this
aproach in the next paragraph. 

\subsection{Morawetz estimates}
\label{sec:morawetz}

For $(x,y)\in \R^d$ and $\mu>0$, we denote by $Q(x,y,\mu)$ a dilation of
the unit cube
centered in $(x,y)$,
\begin{equation*}
  Q(x,y,\mu) = (x,y)+[-\mu,\mu]^d.
\end{equation*}
\begin{proposition}\label{prop:morawetz}
  Let $u\in C(\R;Z)$ be as in Theorem~\ref{theo:cauchy}. For every
  $\mu>0$, there exists $C_\mu>0$ such that
  \begin{align*}
    \left\| |\nabla_y|^{\frac{4-d}{2}}R\right\|_{L^2_{ty}(\R\times
      \R^{d-1})}^2 + \int_{\R}& \(\sup_{(x_0,y_0)\in
    \R^d}\iint_{Q(x_0,y_0,\mu)}|u(t,x,y)|^2dxdy \)^{\si+2} dt  \\
&\le
  C_\mu\sup_{t\in \R}\|u(t)\|_{H^1_{xy}}^4\lesssim \|u_0\|_Z^4,
  \end{align*}
where 
\begin{equation*}
  R(t,y) = \int_{-\infty}^{+\infty} |u(t,x,y)|^2dx
\end{equation*}
is the marginal of the mass density. 
\end{proposition}
\begin{proof}
  We resume the computations from \cite[Section~5]{AnCaSi-p}, and
  simply recall the main steps. 
\smallbreak

To shorten the notations, we set $z=(x, y)$. Following
\cite{GiVe10}, we write that if $u$ is a solution
to \eqref{eq:nls}, then we have 
\begin{equation}\label{eq:cons_laws}
\left\{\begin{aligned}
&\d_t\rho+\diver J=0\\
&\d_tJ+\diver\(\RE(\nabla\bar u\otimes\nabla u)\)+\frac{\sigma}{\sigma+1}\nabla\rho^{\sigma+1}
+\rho\nabla V=\frac14\nabla\Delta\rho,
\end{aligned}\right.
\end{equation}
where $\rho(t, z):=|u(t, z)|^2$ and $ J(t, z):=\IM(\bar u\nabla u)(t, z)$.
Let us define the virial potential
\begin{equation*}
I(t):=\frac12\iint_{\R^d\times\R^d}\rho(t, z)a(z-z')\rho(t, z')\,dzdz'=\frac12\langle\rho, a\ast\rho\rangle,
\end{equation*}
where $a$ is a sufficiently smooth even weight function which will be
be eventually a function of $y$ only. Here $\langle\cdot,\cdot\rangle$
denotes the scalar 
product in $L^2(\R^d)$. By using \eqref{eq:cons_laws}, we see that the time
derivative of $I(t)$ reads 
\begin{equation}\label{eq:mor_act}
\frac{d}{dt}I(t)=-\langle\rho,\nabla a\ast J\rangle=\iint\rho(t, z')\nabla a(z-z')\cdot J(t, z)\,dz'dz=:M(t),
\end{equation}
where $M(t)$ is the Morawetz action. By using again the balance laws
\eqref{eq:cons_laws} we have 
\begin{equation}\label{eq:Moraw}
\begin{aligned}
\frac{d}{dt}M(t)=&-\langle J,\nabla^2a\ast J\rangle+\langle\rho,\nabla^2a\ast\RE(\nabla\bar u\otimes\nabla u)\rangle
+\frac{\sigma}{\sigma+1}\langle\rho,\Delta a\ast\rho^{\sigma+1}\rangle
  \\
&-\langle\rho,\nabla a\ast(\rho\nabla V)\rangle-\frac14\langle\rho,\Delta a\ast\Delta\rho\rangle\\
=&-\langle\IM(\bar u\nabla u), \nabla^2a\ast\IM(\bar u\nabla u)\rangle
+\langle\rho,\nabla^2a\ast(\nabla\bar u\otimes\nabla u)\rangle\\
&+\frac{\sigma}{\sigma+1}\langle\rho,\Delta a\ast\rho^{\sigma+1}\rangle
-\langle\rho,\nabla a\ast(\rho\nabla V)\rangle
-\frac14\langle\rho,\Delta a\ast\Delta\rho\rangle,
\end{aligned}
\end{equation}
where in the second term we dropped the real part because of the
symmetry of $\nabla^2a$ (here, the notation $\nabla^2a\ast\RE(\nabla\bar
u\otimes\nabla u)$ stands for
$\sum_{j,k}\partial^2_{jk}a\ast\RE(\partial_k\bar u\partial_j u)$). Leaving out the details presented in
\cite{AnCaSi-p} and \cite{TzVi-p}, the computation shows that if
$\nabla^2 a$ is non-negative and if $a$ depends on $y$ only (so we
have $\nabla a(z_1)\cdot \nabla V(z_2)=0$ for all $z_1,z_2\in \R^d$),
then we have:
\begin{equation}\label{eq:dtm2}
\frac{d}{dt}M(t)\ge \frac12\langle\nabla_y\rho,
\Delta_ya\ast\nabla_y\rho\rangle 
+\frac{\sigma}{\sigma+1}\langle\rho, \Delta_ya\ast\rho^{\sigma+1}\rangle.
\end{equation}
Now we consider two choices for the weight $a$. First, for $a(y)=|y|$,
we have indeed $\nabla^2 a\ge 0$ as a symmetric matrix, and for $d\ge
3$, 
 $\Delta_ya(y)=\frac{d-2}{|y|}$: it is, up to a multiplicative
 constant, the integral kernel of the operator  
$(-\Delta_y)^{-\frac{d-2}{2}}$, that is,
\begin{equation*}
\((-\Delta_y)^{-\frac{d-2}{2}}f\)(y)=\int_{\R^{d-1}}\frac{c}{|y-y'|}f(y')\,dy'.
\end{equation*}
Thus, by recalling $z=(x, y)$, we obtain
\begin{multline*}
\iint_{\R^d\times\R^d}\frac{1}{|y-y'|}\nabla_y\rho(t, z')\cdot\nabla_y\rho(t, z)\,dz'dz\\
=\iiint_{\R\times\R\times\R^{d-1}}\nabla_y\rho(t, x, y)
\cdot\nabla_y(-\Delta_y)^{-\frac{d-2}{2}}\rho(t, x', y)\,dxdx'dy.
\end{multline*}
Hence, if we define the marginal of the mass density
\begin{equation*}
R(t, y):=\int_{\R}\rho(t, x, y)\,dx,
\end{equation*}
the last integral also reads
\begin{equation*}
\int_{\R^{d-1}}\left||\nabla_y|^{\frac{4-d}{2}}R(t, y)\right|^2\,dy.
\end{equation*}
We now plug this expression into \eqref{eq:dtm2} and we integrate in
time. Furthermore,  the second term in the right hand side in
\eqref{eq:dtm2} is positive. We then infer 
\begin{equation}\label{eq:mor_est}
\int_{-T}^T\int_{\R^{d-1}}\left||\nabla_y|^{\frac{4-d}{2}}R(t,
  y)\right|^2\,dydt\le C\sup_{t\in[-T, T]}|M(t)|.
\end{equation}
Furthermore, with our choice of the weight $a$, we have
\begin{equation*}
|M(t)|=\left|\iint\rho(t, z')\frac{y-y'}{|y-y'|}\cdot\IM(\bar u\nabla_yu)(t, z)\,dz'dz\right|
\le\|u_0\|_{L^2(\R^d)}^3\|\nabla_yu(t)\|_{L^2(\R^d)},
\end{equation*}
hence the first part of Proposition~\ref{prop:morawetz} in the case
$d\ge 3$. In the case $d=2$, the choice $a(y)=|y|$ leads to $a''(y) =
2\delta_0$, and the conclusion remains the same.
\smallbreak

Now, as in \cite{TzVi-p}, consider the weight $a(y)=\<y\>$: we still
have $\nabla^2 a\ge 0$. Resume \eqref{eq:Moraw}: the computations from
  \cite{TzVi-p,PlVe09}
 yield a rearrangement of the terms so that instead of
  \eqref{eq:dtm2}, we now have
\begin{equation*}
\frac{d}{dt}M(t)\ge \frac{\sigma}{\sigma+1}\langle\rho,
\Delta_ya\ast\rho^{\sigma+1}\rangle. 
\end{equation*}
The right hand side is equal to 
\begin{equation*}
  \frac{\si}{\si+1}\iint\iint |u(t,x_1,y_1)|^2 \Delta a (y_1-y_2)
  |u(t,x_2,y_2)|^{2\si+2}dx_1dy_1dx_2dy_2.
\end{equation*}
Following \cite{TzVi-p}, we note that
\begin{equation*}
  \inf_{Q(0,0,2\mu)}\Delta_y\(\<y\>\)>0,
\end{equation*}
so the above term is bounded from below by constant times
\begin{equation*}
  \sup_{(x_0,y_0)\in \R^d}\iint_{Q(x_0,y_0,\mu)}
  \iint_{Q(x_0,y_0,\mu)}  |u(t,x_1,y_1)|^2  |u(t,x_2,y_2)|^{2\si+2}dx_1dy_1dx_2dy_2.
\end{equation*}
H\"older inequality yields
\begin{equation*}
  \iint_{Q(x_0,y_0,\mu)} |u(t,x_2,y_2)|^{2\si+2}dx_2dy_2\gtrsim
  \(\iint_{Q(x_0,y_0,\mu)} |u(t,x_2,y_2)|^{2}dx_2dy_2\)^{\si+1}. 
\end{equation*}
Finally, with this second choice for $a$, we still have
\begin{equation*}
  |M(t)|\le \|u_0\|_{L^2_{xy}}^3\|\nabla_y u(t)\|_{L^2_{xy}},
\end{equation*}
hence the result by integrating in time. 
\end{proof}

\subsection{End of the argument}

To prove Theorem~\ref{theo:AC} in the case $d\le 4$, one can resume
the approach followed in \cite[Section~6]{AnCaSi-p} which is readily adapted to
our framework, the only difference being that the function space and
the related set of vectorfields are not the same here. 
\smallbreak

However, as
pointed out in \cite{TzVi-p}, the fact that negative order derivatives
are involved in the first term in Proposition~\ref{prop:morawetz}
makes it delicate to use this term when $d\ge 5$, and requires fine
harmonic analysis estimates in the case $V=0$; it is not clear whether
or not 
these tools can be adapted to the present setting. This is why the second term in
Proposition~\ref{prop:morawetz}, which corresponds to the one
considered in \cite{TzVi-p}, is more efficient then, and allows to
prove Theorem~\ref{theo:AC}  for all $d\ge 2$. 
\smallbreak

The first step stems from \cite{Vi09}:
Theorem~\ref{theo:cauchy} and Proposition~\ref{prop:morawetz}  imply that
\begin{equation*}
  \|u(t)\|_{L^r_{xy}}\Tend t {+\infty} 0,\quad \forall 2<r<\frac{2d}{(d-2)_+}.
\end{equation*}
The end of the proof is presented in \cite{TzVi-p}, and is readily
adapted to our framework: it consists in choosing suitable  Lebesgue
exponents and applying inhomogeneous Strichartz estimates for
non-admissible pairs, which follow in our case from
\cite{AnCaSi-p,FoschiStri}. Since the proof is then absolutely the
same as in \cite{TzVi-p}, we choose not to reproduce it here.

\bibliographystyle{siam}

\bibliography{partiel}

\end{document}